\documentclass[leqno,a4paper]{article}
\usepackage{amsmath,amsthm,amssymb}
\usepackage{eufrak}

\newcommand{\cD}{\mathcal{D}}

\newcommand{\mT}{\mathcal{T}}
\newcommand{\gT}{\textbf{T}}
\newcommand{\cT}{{\widetilde{\textbf{T}}}}

\newcommand{\NN}{{\mathbb N}}
\newcommand{\R}{{\mathbb R}}
\newcommand{\ox}{{\overline{x}}}

\newcommand{\ut}{{\underline{t}}}

\newcommand{\De}{{\Delta}}
\newcommand{\na}{{\nabla}}
\newcommand{\pa}{{\partial}}

\newcommand{\si}{{\sigma}}
\newcommand{\ta}{{\tau}}

\newcommand{\om}{{\omega}}
\newcommand{\Om}{{\Omega}}
\newcommand{\Ga}{{\Gamma}}
\newcommand{\al}{{\alpha}}

\newcommand{\n}{{\underline{n}}}

\newcommand{\ca}{{\widetilde{a}}}

\newcommand{\cp}{{\widetilde{p}}}
\newcommand{\ccu}{{\widetilde{u}}}

\newcommand{\wt}{\widetilde}
\newcommand{\be}{\begin{equation}}
\newcommand{\ee}{\end{equation}}
\newcommand{\ba}{\begin{array}}
\newcommand{\ea}{\end{array}}
\def\vs1{\vspace{1ex}}
\newtheorem{definition}{Definition}[section]
\newtheorem{theorem}{Theorem}[section]
\newtheorem{lemma}[theorem]{Lemma}
\newtheorem{proposition}{Proposition}[section]
\newtheorem{corollary}{Corollary}[section]
\newtheorem{remark}{Remark}[section]
\newtheorem{assumption}{Assumption}[section]
\newtheorem{problem}{Problem}[section]
\numberwithin{equation}{section}
\begin{document}
\title{\bf\normalsize On the reduction of PDE's problems in the half-space, under
the slip boundary condition, to the corresponding problems\\ in the
whole space}
\author{by H.~Beir\~ao da Veiga, F.~Crispo, C.~R.~Grisanti }
\date{}
\maketitle \textit{\phantom{aaaaaaaaaaaaaaaaaaaaaaaaaa}}%

\begin{abstract}
The resolution of a very large class of linear and non-linear,
stationary and evolutive partial differential problems in the
half-space (or similar) under the slip boundary condition is reduced
here to that of the corresponding results for the same problem in
the whole space. The approach is particularly suitable for proving
new results in strong norms. To determine whether this extension is
available, turns out to be a simple exercise. The verification
depends on a few general features of the functional space $\,X\,$
related to the space variables. Hence, we present an approach as
much as possible independent of the particular space $X$. We appeal
to a reflection technique. Hence a crucial assumption is to be in
the presence of flat boundaries (see below).\par%
Instead of stating "general theorems" we rather prefer to illustrate
how to apply our results by considering a couple of interesting
problems. As a main example, we show that the resolution of a class
of problems for the evolution Navier-Stokes equations under a slip
boundary condition can be reduced to that of the corresponding
results for the Cauchy problem. In particular, we show that sharp
vanishing viscosity limit results that hold for the evolution
Navier-Stokes equations in the whole space can be extended to the
boundary value problem in the half-space. We also show some
applications to non-Newtonian fluid problems.

\vspace{.2cm}

{\bf Mathematics Subject Classification} 35Q30, 76D05, 76D09.

\vspace{.2cm}

{\bf Keywords.} Navier slip boundary conditions, Navier-Stokes
equations, inviscid limits, non-Newtonoan fluids.

\end{abstract}

\bibliographystyle{amsplain}

\section{Introduction.}
In the following we consider the slip boundary condition
\begin{equation}
\left\{
\begin{array}{l}
u\cdot \n=\,0,\\
\ut \times \,\n=\,0\,,
\end{array}
\right.
\label{boundary-ho}
\end{equation}
where $\,\n\,$ is the outward unit normal to the boundary $\,\Ga\,$,
$\,\ut =\,\mT\cdot\,\n\,$ is the stress vector, and
\begin{equation*}
\mT =-\pi\,I+\,\frac{\nu}{2} \,(\nabla u+\nabla u^T)
\end{equation*}
is the stress tensor. On flat portions of the boundary
\eqref{boundary-ho} simply reads
\begin{equation}
\left\{
\begin{array}{l}
u\cdot \n=\,0,\\
\om \times \,\n=\,0\,,
\end{array}
\right. \label{bcns}
\end{equation}
where $\,\om=\, \na \times \,u\,$. In the general case they differ
only by lower order terms. In the sequel we appeal to
\eqref{boundary-ho}, since our results are proved in the case of flat boundaries.\par%
The literature on slip boundary conditions is particularly vast. The
boundary conditions \eqref{boundary-ho} were proposed by Navier, see
\cite{Navier}. A first mathematical study is due to Solonnikov and
\v{S}{c}adilov in the pioneering paper \cite{so-sca}. In \cite{bv1},
a quite general and self-contained presentation is given. In these
two references regularity results up to the boundary are considered.
See also \cite{BV-annalen} and \cite{bae-jin}, where the regularity
problem is considered in the half-space. We also refer to
\cite{galdi-layton}, \cite{gr}, and \cite{watanabe}.

\vspace{0.2cm}

As already announced in the abstract, as a main example, we turn the
resolution and properties of the evolution Navier-Stokes equations
\begin{equation}
\left\{
\begin{array}{l}
\partial_t\,u +\,(u\cdot\,\nabla)\,u -\,\nu\,\Delta\,u+\,\nabla\,p=\,0,\\
\na \cdot\,u =\,0\,,\\
u(0)=\,a(x)\,,
\end{array}
\right. \label{1;4}
\end{equation}
under the slip boundary condition into the corresponding results for
the Cauchy problem, all at once. Often, sharp results for the Cauchy
problem are known, but counterparts under boundary conditions are
known only in weaker forms. This is the situation concerning the
convergence of the solutions $\,\ccu^\nu\,$ of the Navier-Stokes
equations under the slip boundary condition to the solution of the
Euler equations, as the viscosity $\,\nu\,$ goes to zero. As an
application we will consider this problem. We show that if the
initial data $\,\ca\,$ is given in a suitable functional space
$\,X\,$ then the solutions $\,\ccu^\nu\,$ to the Navier-Stokes
problem belong to $\,C(\,[0,\,T\,]; X\,)\,$ (T.Kato's persistence
property). Further, as the viscosity $\,\nu\,$ goes to zero,
$\,\ccu^\nu\,$ converges in the strong, and uniformly in time,
$\,C(\,[0,\,T\,]; X\,)\,$ norm to the solution of the Euler
equations under the zero-flux boundary condition. We may also assume
that the initial data $\,\ca^\nu\,$ to the Navier-Stokes problem
depends on the viscosity, and converges in $\,X\,$ to some
$\,\ca\,$. See the Theorem \ref{teoka} below, where
$\,X=\,H^{l,\,2}_\si(\R^3_+)\,$. By using the same ideas one can
also approach stationary problems. Clearly in this case a force term
must be included. See the Theorem \ref{teoapplNN} (stationary
problem, for shear-thickening fluids),
and the Theorem \ref{teopev} (evolution problem for shear-thinning fluids).\par%
It is worth noting that, either in the evolution or in the
stationary case, one can solve several problems besides the
Navier-Stokes equations, or get different properties for the
solutions.\par%
Since we are mainly interested in applications to incompressible
fluids, we assume from the very beginning that spaces $\,X\,$
consist of divergence free vector fields. Further, as a rule, functions
labeled by a tilde are defined in the half-space $\,\R^3_+\,$.\par%
The reflection technique, followed here, applies in the presence of
flat boundaries, for instance cubic domains (like that used in
references \cite{bvcr} and \cite{bvcr2}), 3-D strips, or the
half-space $\, \R^3_+ =\,\{\,x \in \R^3\,:\,x_3
>\,0\,\}\,.$ To fix ideas, we mainly refer to the half space case and set%
$$
\Ga =\,\{\,x \in \R^3\,:\,x_3 =\,0\,\}\,.%
$$
Note that $\Ga$ is called ``boundary", even when we consider
functions defined in the whole space. The unit normal to $\Ga$
(outward, with respect to $\R^3_+ \,$) is denoted by $\n\,.$ Traces
on $\Ga$ ``from above" and ``from below" mean, respectively, from
the $x_3>\,0\,$ side, and from the $x_3< \,0\,$ side.\par%

\section{Functional framework.}
In the following, by ``boundary value problem" we always refer to
\eqref{bcns}. Roughly speaking, our main interest is showing that
results hold for the boundary value problem in the framework of a
given functional space $X(\R^3_+)$, \emph{if they hold for the
corresponding problem in $X(\R^3)$}. Due to the general nature of
this last hypothesis, it would be restrictive to dwell upon the
spaces $\,X\,$. This is not difficult, since the main lines depend
only on some general properties of $X$. Some basic assumptions on
the functional spaces $\,X\,$ are obvious from the context, for
example, elements are locally integrable functions, together with a
certain number of partial derivatives. To fix ideas, we suggest that
the reader thinks of $\,X(\R^3)\,$ as being a Sobolev space
$W^{s,\,p}_{\si}(\R^3)\,$ ($\,\si$ stands for ``divergence free").

\vspace{0.2cm}

Concerning the link between the spaces $\,X(\R^3)\,$ and
$\,X(\R^3_+)\,,$ we assume the following, standard, properties:
Restrictions $\,u_+\,$ to $\R^3_+$ (resp. $\,u_-\,$ to $\R^3_-\,$)
of elements $\,u\,\in \,X(\R^3)$ belong to $X(\R^3_+)\,$ (\, resp.
$X(\R^3_-)\,)\,$. Moreover, if  $\,u\,\in \,X(\R^3)\,,$ the norm of
$\,u\,$ in this space is equivalent to the sum of the norms of the
restrictions $\,u_+\,$ and $\,u_-\,$ in the above corresponding
spaces.  Norms in $\,X(\R^3_+)\,$ and in
$\,X(\R^3_-)\,,$ are defined ``symmetrically".\par%
It is in general false (for spaces used in PDE's) that $u_+ \in
\,X(\R^3_+)\,,$ and $u_- \in \,X(\R^3_-)\,$ implies $\,u\,\in
\,X(\R^3)\,$.
In the sequel we assume the following typical situation.%
\begin{assumption}
The functional spaces $\,X\,$ consist of divergence free vector
fields. Further, in correspondence to a given $\,X(\R^3)\,$ space,
an integer $\,l_0 =\,l_0(X)\,$ exists so that partial derivatives
of elements of $X(\R^3)$ have traces on $\,\Ga\,$, in the usual
sense, if  and only if their order is less or equal to $l_0$.
Moreover, $\,u\,$ belongs to $\,X(\R^3)\,$ if and only if its
restrictions satisfy $\,u_+ \in \,X(\R^3_+)\,,$ $\,u_- \in
\,X(\R^3_-)\,,$ and traces of homonymous derivatives coincide,
from above and from below, up to the order
$\,l_0\,$. %
\label{existlo}
\end{assumption}
Assumption \ref{existlo} avoids some singular cases like Sobolev
spaces $\,W^{s,\,p}\,$, if $\,s -\,\frac1p\,$ is an integer. In this
case, and in similar ones, the theory presented in the sequel needs
some adaptation, not considered here.

\vspace{0.2cm}

For convenience, we often use the symbol $\,X^l\,$ to specify the
largest order of the derivatives that appear in the definition of
$X\,$. Clearly, $\,l_0 \leq \,l\,$. For instance, if $\,X^l=\,H^l
=\,H^{l,\,2}\,$, $\,l$ integer, then $ l_0=\,l -\,1\,.$\par%
\section{Some basic results.}
We set
$$
\ox=\,(x_1,\,x_2,\,-\,x_3\,)
$$
and start this section by introducing the following definition.
\begin{definition}
Let $p$ and $v$ be an arbitrary scalar field and an arbitrary vector
field in $\R^3\,$. We set, for each $x \in \,\R^3\,$,
$(\gT\,p)(x)=\,p(\ox)\,,$ and
$$
(\gT\,v)(x)=\,(\,\gT\,v_1(x),\,\gT\,v_2(x),\,-\,\gT\,v_3(x)\,)=\,
(\,v_1(\ox),\,v_2(\ox),\,-\,v_3(\ox)\,)\,.
$$
We also define $\,(\gT\,v)(x)\,,$ for  $x \in \,\R^3_-\,$, if
$\,v\,$ is defined in $\,\R^3_+\,.$ \label{ttrans}
\end{definition}
\begin{proposition}
One has
\begin{equation}
\left\{
\begin{array}{l}
\na\,(\gT\,p)=\,\gT\,(\na\,p)\,,\\
\na \cdot\,(\gT\,v)=\,\gT\,(\na \cdot\,v\,)\,,\\
\na \times\,(\gT\,v)=\,-\,\gT\,(\na \times\,v\,)\,,\\
\De \,(\gT\,v)=\,\gT\,(\De\,v)\,,\\
(\,(\gT\,v)\cdot\,\na\,)(\gT\,v)=\,\gT\,(\,(v\,\cdot \,\na)\,v\,)\,,\\
\pa_t\,(\gT\,v)=\,\gT\,(\,\pa_t\,v)\,,
\end{array}
\right.\label{ttrans}
\end{equation}
where the terms in the left hand side are taken in $x$ and the
corresponding terms in the right hand side in $\ox$. Note that
$\,\gT^{\,2}=\,I\,$.
\label{tteo}%
\end{proposition}
The proof is left to the reader.\par%
\begin{lemma}
Let $v$ be a vector field in $\,\R^3\,$. If $\,\gT\,v=\,v\,$ then
$\,v_3(x_1,\,x_2,\,0)=\,0\,$. If $\,\gT\,v=\,-\,v\,$ then
$\,v_j(x_1,\,x_2,\,0)=\,0\,,$ for $\,j=\,1,\,2.$ In particular, if
$\,\gT\,v=\,v\,$ then $\,(\na\times\,v\,)_j(x_1,\,x_2,\,0)=\,0\,,$ for $j=\,1,\,2\,.$%
\label{prip}
\end{lemma}
In fact, from $\,\gT\,v=\,v\,$ it follows that
$$
(\gT\,v)_3(x_1,\,x_2,\,x_3)=\,v_3(x_1,\,x_2,\,x_3)\,.
$$
On the other hand,
$$
\,(\gT\,v)_3(x_1,\,x_2,\,x_3)=-\,(\gT\,v_3)(x_1,\,x_2,\,x_3)=-\,v_3(\ox)\,.
$$
Consequently, $v_3=\,0\,$ for $\,x_3=\,0\,$. The second statement
follows by a similar argument. The last statement follows from the
relation \eqref{ttrans}$_3$ for $\,\gT\,v=\,v\,$
$$
\na \times\,v=\,-\,\gT\,(\na \times\,v\,)\,.
$$
\begin{corollary}
If $\,v=\,\gT\,v\,$ then $v$ satisfies the boundary conditions \eqref{bcns}.%
\label{quatro}
\end{corollary}
\section{An explanatory interlude.}%
Before going to the next section it looks useful to justify it by
appealing to an example. To this end we appeal to the solution to
the evolution Navier-Stokes equations. One has the following result.
\begin{proposition}
If $u$ is a solution to the Cauchy problem \eqref{1;4} with initial
data $a$, then $\gT\,u$ is a solution to the Cauchy problem with initial data $\gT\,a$.\par%
If $\,a \in \,X(\R^3)\,$, $\,a=\,\gT\,a\,,$ and $\,u\,$ is the
unique solution to the Cauchy problem \eqref{1;4}, then $u=\,\gT\,u$
for each
$t \in [\,0,\,T\,]\,$.\par%
Moreover, the restriction of $u(t)\,$ to $\R^3_+$ solves the initial
boundary-value problem in the half space, under the boundary
condition \eqref{bcns}.
\label{ttea}%
\end{proposition}
\begin{proof}
The proof of the first part follows by applying the linear operator
$\gT$ to each of the single equations \eqref{1;4}, and by appealing
to
 Proposition \ref{tteo}. The proof of the second assertion follows
from the first part and from the uniqueness of the strong solutions
to the Cauchy problem. The third claim follows by appealing to
Corollary \ref{quatro}.
\end{proof}
The above result essentially shows that if  $\,a \in \,X(\R^3)\,$
satisfies $\,a=\,\gT\,a\,,$ then statements that hold for the Cauchy
problem (like inviscid limit results) also hold in the half-space,
for the boundary value problem \eqref{bcns} with initial data
 $\,\ca=\,a_{|\R^3_+}\,.$ However, in considering directly the boundary value problem, the
initial data $\,\ca\,$ is defined in the half-space. Hence, we have
to study how general $\,\ca \in \,X(\R^3_+)\,$ may be, so that it is
the restriction to $\R^3_+\,$ of some  $\,a \in \,X(\R^3)\,$ for
which $a=\,\gT\, a$. In other words, we must express the implicit
constraint $a=\,\gT\,a$ in terms of explicit assumptions on
$\,\ca\,$. This leads to the following problem.
\begin{problem}
Given an arbitrary $\,\ca\, \in  X(\R^3_+)$, which satisfies the
slip boundary condition, look for necessary and sufficient
conditions so that $\,\ca\,$ is the restriction to $\,\R^3_+ \,$ of
some $\,a \,\in X(\R^3) \,$ for which $\,\gT\,a=\,a\,.$%
\label{melb}
\end{problem}
This is the subject of the next section. For instance, we will see
that, if $\,l_0 <\,3\,,$ the answer is always positive (as, for
instance: for $\,H^3_\si \,$; for $\,W^{3,\,p}_\si \,$ and
arbitrarily large $\,p\,$). On the contrary, $\,H^4_\si\,$ and
$\,W^{4,\,p}_\si\,$ require compatibility conditions.
\section{The compatibility conditions, and ``fitting" on the boundary.}%
The main subject of this section is to prove that the conditions
\eqref{numas} below are the right \emph{compatibility conditions}.%
\begin{assumption}
Let $\,\ca=\,(\ca_1,\,\ca_2,\,\ca_3) \in X^l(\R^3_+)\,.$ If
$\,l_0=\,l_0(X^l) \geq \,3\,,$ then for each \emph{odd} integer $\,k
\in [3,\,l_0\,]\,$, the partial derivatives $\,\pa^k_3
\,\ca_j\,,$ $\,j=\,1,\,2\,,$ vanish on $\,\Ga\,$:%
\begin{equation}
\pa^k_3\,\ca_1(x_1,\,x_2,\,0)=\,\pa^k_3\,\ca_2(x_1,\,x_2,\,0)=\,0\,.%
\label{numas}
\end{equation}
\label{sumo}
\end{assumption}
In the next section we show that these conditions are independent.\par%
We introduce the following convention. If $\,g\,$ is defined in
$\,\R^3\,$ (or in $\,\R^3_+ \cup \,\R^3_-\,$), we say that $\,g\,$
\emph{fits on $\Ga$}, or just \emph{fits}, if the traces on $\Ga$,
from above and from below, coincide. Furthermore, the expression
\emph{fits by zero} means that both traces vanish on $\Ga$. Clearly,
these definitions are meaningful
only when the traces from both sides exist. If not, ``fitting" is not defined.\par%
The following theorem is the main result in this section.
\begin{theorem}
Let $\ca \in \,X^l(\R^3_+)\,$ satisfy the slip boundary conditions.
Then there is an
$$
a \in \,X^l(\R^3)
$$
such that $\,\gT\,a=\,a\,$ in $\,\R^3\,,$ and
$$
a_{|\,\R^3_+}=\,\ca\,,
$$
if and only if $\ca$ satisfies the assumptions \ref{numas}.
\label{tefu}
\end{theorem}
For instance, if $l_0=3$ or $l_0=4$, the assumption
$$
\,\pa^3_3 \,\ca_1 =\,\pa^3_3 \,\ca_2=\,0 \quad \textrm{on} \quad \Ga
\,,
$$
guaranties that the traces, from above and from below, of any
derivative of $\,a\,$ of order less or equal to $\,l_0\,$ ($\,a\,$
defined by \eqref{defas}), fit on $\Ga$.%

\vspace{0.2cm}

In order to prove Theorem \ref{tefu}, we define in $\R^3\,$ the
\emph{mirror-extension} $\,\cT \,\ca\,$ of $\,\ca\,$, defined on
 $\R^3_+\,$, by the equation
\begin{equation}
a=\,\cT \,\ca\,:=\left\{ \begin{array}{ll}%
\ca & \textrm{in $\R^3_+$\,,}\\
\gT \,\ca & \textrm{in $\R^3_-$\,.}%
\end{array} \right.%
\label{difas}
\end{equation}
It is obvious that $a=\,\cT \,\ca\,$ is the unique extension of
$\ca$ for which $\,a=\,\gT\,a\,.$ The point here is whether or not
$\,\ca \in X(\R^3_+)$ implies that $\,a=\,\cT \,\ca\,$, belongs to
$\,X(\R^3)\,$. By assumption \ref{existlo}, this is equivalent to
proving the coincidence of the traces on $\Ga$ ``from above and from
below" of all partial derivatives of $\,a\,$ up to order $\,l_0\,$.
We will show that this is equivalent to the conditions on the traces
(obviously from above)
of the derivatives of $\,\ca\,$ referred to in assumption \ref{sumo}.\par%
For convenience we write \eqref{difas} in the more explicit form
$$
a_j(x_1,\,x_2,\,x_3)=\, \ca_j(x_1,\,x_2,\,x_3) \quad \textrm{in}
\quad \R^3_+\,, \quad \textrm{if\,}  \,j=\,1,\,2,\,3\,;
$$
\begin{equation}
a_j(x_1,\,x_2,\,x_3)= \left\{ \begin{array}{ll}%
\ca_j(x_1,\,x_2,\,-\,x_3) & \textrm{in $\R^3_-$\,, if
$j=\,1,\,2\,,$}\\
-\,\ca_3(x_1,\,x_2,\,-\,x_3) & \textrm{in $\R^3_-$\,, if $j=\,3\,$,}
\end{array} \right.
\label{defas}
\end{equation}
and also  the slip boundary condition \eqref{bcns} in the explicit
form
\begin{equation}
\left\{ \begin{array}{ll}%
\ca_3(x_1,\,x_2,\,0)=\,0\,,\\
\pa_3 \,\ca_j(x_1,\,x_2,\,0)=\,\pa_j \,\ca_3(x_1,\,x_2,\,0)\,,
\qquad \textrm{for} \quad j=\,1,\,2\,.
\end{array} \right.
\label{asstil}
\end{equation}

\vspace{0.2cm}

It is worth noting that if a function fit (respectively, fit by
zero), then its tangential derivative of any order also fit
(respectively, fit by zero). Hence, the ``fitting problem" for a
partial derivative $\,\pa^k_\ta \,\pa^m_3 \,a_j\,$ is reduced to the
same problem for the pure normal derivative $\,\pa^m_3 \,a_j\,$. For
convenience, we put in evidence this result.
\begin{lemma}
If a partial derivative $\,\pa^m_3\,a_j\,$ fits (resp. fits by zero)
on $\,\Ga\,$, then any partial derivative $\,\pa^k_\ta
\,\pa^m_3\,a_j\,$ fits (respectively, fits by zero).
\end{lemma}
The next result follows easily from the definitions.
\begin{proposition}
Let $\ca$ be given in $\R^3_+$, and let $a=\,\cT \,\ca$ be the
mirror-extension of $\ca$ to $\R^3$. Then:\par%
a) Partial derivatives of $a_3$, of odd order in the normal
direction, fit on $\Ga$. Partial derivatives of $a_j\,,$ for
$\,j=\,1,\,2\,$, of even order in the normal direction, fit on
$\Ga$.\par%
b) Partial derivatives of $a_3$, of even order in the normal
direction, and partial derivatives of $a_j\,,$ for $\,j=\,1,\,2\,$,
of odd order in the normal direction, fit on $\Ga$ if and only if they fit by zero.%
\label{banilias}
\end{proposition}
 Proposition \ref{banilias} shows that a necessary and sufficient
condition for the resolution of problem \ref{melb} is the fitting by
zero of the partial derivatives considered in part b). However these
conditions are not independent. We start by proving the following
result.
\begin{proposition}
Compatibility conditions are not required for derivatives of order
less than or equal to two. In particular if $\,l_0(X) \leq\,2\,.$
\label{menosba}
\end{proposition}
\begin{proof}
 We have to show that
\begin{equation}
a_3(x_1,\,x_2,\,0 )=\,\pa_3\,a_j(x_1,\,x_2,\,0 )=\,
\pa^2_3\,a_3(x_1,\,x_2,\,0 )=\,0\,,\quad \textrm{for} \quad j=\,1,\,2\,.%
\label{asstel}
\end{equation}
The two first assertions follow easily from \eqref{asstil}. On the
other hand, due to the divergence free property, one has
\begin{equation}
-\,\pa^2_3 \,a_3=\,\sum_{j=\,1,\,2} \pa_j \,\pa_3\,a_j \,.%
\label{faltavas}
\end{equation}
By results already shown, both terms on the right hand side of the
above equation have zero trace on $\Ga$.\par%
\end{proof}
\begin{proposition}
Let $\,\ca\,$ be a divergence free vector field defined in
$\,\R^3_+\,$ and let $a=\,\cT \,\ca$ be the mirror-extension of
$\ca$ to $\R^3$. If, for some odd integer $k$,
\begin{equation}
\pa^k_3 \,\ca_j=\,0 \quad \textrm{on}\quad \Ga\,, \quad \textrm{for} \quad j=\,1,\,2\,,%
\label{boh}
\end{equation}
then
\begin{equation}
\,\pa^{k+\,1}_3 \,a_3=\,0\,,%
\label{bah}
\end{equation}
in $\Ga\,$. In other words if $\,\pa^k_3 \,\ca_j \,$ fits by zero,
for $\,j=\,1,\,2\,,$ then $\,\pa^{k+\,1}_3 \,a_3\,$ fits by zero.%
\label{tres}
\end{proposition}
\begin{proof}
By appealing to the divergence free property we get
\begin{equation}
-\,\pa^{k+\,1}_3 \,a_3(x)= \left\{ \begin{array}{ll}%
\sum_{j=\,1}^{2} \, (\pa_j\,\pa^k_3\,\ca_j)(x)& \textrm{if $x_3 >\,0$\,,}\\
(-1)^{k+\,1}\,\sum_{j=\,1}^{2} \, (\pa_j\,\pa^k_3\,\ca_j)(\ox)& \textrm{if $x_3 <\,0$\,.}%
\end{array} \right.%
\end{equation}
This leads to the thesis.
\end{proof}%
Note that the result also holds for even values of $k$. Proposition
\ref{tres}, together with results already established, prove Theorem
\ref{tefu}.
\section{Independence of the compatibility conditions.}
We already have shown that \eqref{boh} and \eqref{bah} are necessary
conditions for fitting. Proposition \ref{tres} shows that these
conditions are not independent. We wonder whether the subset in
assumption \ref{sumo} is minimal. We show here that this is the
case, by constructing a divergence free vector field $\,v \in
\,C^{\infty}(\R^3)\,$, with compact support contained in a sphere
centered in the origin, and with radius arbitrarily small, which
satisfies the boundary conditions and all the compatibility
conditions up to an arbitrary odd order $\,n-\,1\,$, but which do
not satisfy the compatibility condition of order $\,n\,$. This shows
not only that the last compatibility condition is needed (this was
already known), but also that it does not follow from the set of all
the previous (lower order) compatibility conditions, together with
the boundary
conditions and the divergence free property.\par%
Let $\,n \geq \,2\,$ be an integer, fix a scalar field $\rho\in
C_0^{\infty}(B(0,1))\,,$ and define the vector field
$\,w=\,x^{n+\,1}_3 \,(1,\,1,\,0)\,$ in $\,\R^3\,$. It is easy to
check that the divergence free vector field $\,v=\,\na \times
(\rho\,w) \,$ vanishes on the boundary, together with any partial
derivative of order less or equal to $\,n-\,1\,$. In particular,
$\,v\,$ satisfies our boundary conditions, and
\begin{equation}
\pa^k_3\,v_1(x_1,\,x_2,\,0)=\,\pa^k_3\,v_2(x_1,\,x_2,\,0)=\,0\,,%
\label{numas-2}
\end{equation}
on $\Ga\,$, for each $\,k<\,n\,$. However%
\begin{equation}
\pa^n_3\,v(x_1,\,x_2,\,0)=\,\rho\,(n+1)!\,(-1,\,1,\,0)\,,%
\label{nimes}
\end{equation}
shows that $\pa^n_3\,v_1$ and $\pa^n_3\,v_2$ do not vanish on $\Ga$.
Consider an odd value $n \geq \,3$. Then, by \eqref{numas-2}, the
compatibility condition \eqref{numas} is satisfied up to order
$\,n-\,1$. In particular it holds for all odd $k\,,$ up to order
$n-\,2$ included. But the last compatibility condition does not
hold, as follows from \eqref{nimes}.
\section{On a class of solutions to the Navier-Stokes equations in $\R^3\,$.
The ``abstract" theorem \ref{abstract}.}%
Thanks to Proposition \ref{ttea} and Theorem \ref{tefu} we may
extend to boundary value problems many properties which hold for the
Cauchy problem. The same arguments hold for stationary problems,
where a force field $\wt f$ appears in place of $\ca$. Actually, if
we want to extend properties from a problem $P=P\,(\,\R^3)$, in the
whole space, to the corresponding slip boundary value problem $\wt
P=\wt P\,(\,\R^3_+)$, in the half-space, we merely have to check
that: a) $\ca\,$ (respectively $\wt f$) satisfies the assumptions
\ref{numas}, and the boundary conditions \eqref{bcns}; b) if $u$ is
a solution of problem $P\,(\,\R^3)\,$ then $\gT\, u$ is solution of
the
same problem; c) the solution $\,u\,$ is unique (actually, this point may be bypassed).\par%
The Theorem \ref{abstract} below follows the above picture. This
theorem is the foundation of many possible applications to the
Navier-Stokes equations, in particular that (inviscid limit) considered by us in the next section.\par%
Note that the preliminary hypotheses below consist in assuming that a very basic result holds in $\R^3\,.$\par%
We suppose that the functional space $\,X^l(\R^3)$ satisfies the
assumption \ref{existlo}.
\begin{theorem}
Preliminary $\,\R^3\,-$hypotheses: for each $\,a \in \,X^l(\R^3)\,$
the Cauchy problem \eqref{1;4} admits a unique solution $\,u\,\in
C(\,[0,\,T];
\,X^l(\R^3)\,)\,,$ for some positive $T=\,T(a)$.\par%
The $\,\R^3_+\,$ result: Assume that the initial data $\,\ca \in
X^l(\R^3_+)\,$ satisfy the boundary conditions \eqref{bcns} and the
compatibility conditions described in assumption \ref{sumo}. Then
the initial-boundary value problem \eqref{1;4}, \eqref{bcns} admits
a (unique) solution $\,\ccu\,\in C(\,[0,\,T]; \,X^l(\R^3_+)\,)\,$ in
$\,[0,\,T]\,$.\par%
More precisely, the solution $\,\ccu\,$ is constructed as follows.
Given $\,\ca \in X^l(\R^3_+)\,$ as above, we define $\,a=\,\cT \,\ca
\in X^l(\R^3)\,$ as being the mirror-extension of $\,\ca\,$ to
$\,\R^3\,$ (see \eqref{difas} below). Furthermore, let $\,u\,\in
C(\,[0,\,T]; \,X^l(\R^3)\,)\,$ be the solution to the Cauchy problem
with the initial data $\,a \,$. Then, for each $\,t \in [0,\,T]\,,$
the above solution $\,\ccu(t)\,$ is simply the restriction of
$\,u(t)\,$ to the half-space $\,\R^3_+\,$. In particular,
\begin{equation}
 c\,\|u(t)\|_{X(\R^3)} \leq \,\|\ccu(t)\|_{X(\R^3_+)}
 \leq\,\|u(t)\|_{X(\R^3)}\,,
\label{nomas}
\end{equation}
where $\,c=\,c(l)\,$, is a positive constant.
\label{abstract}%
\end{theorem}
\begin{proof}
The Theorem \ref{abstract} follows immediately from the Theorem
\ref{tefu} together with Proposition \ref{ttea}. Indeed the
hypotheses on the data $\ca$ in Theorem \ref{abstract} are the same
as in Theorem \ref{tefu}. From this last theorem we get $a \in
\,X^l(\R^3)\,$ and $a=\,\gT\,a\,$. Therefore the assumptions on $a$
in Proposition \ref{ttea} are satisfied and we get that the
restriction of $u(t)\,$ to $\R^3_+$ solves the initial
boundary-value problem in the half space.
\end{proof}
The link between the solutions $\,\ccu\,$ and $\,u\,$ shows that
\eqref{nomas} holds, in general, for any other $\,X_1$-norm
satisfying the assumption \ref{existlo}, provided that $\,l_0(X_1)
\leq\,l_0(X)\,$.\par%
Concerning time instant $\,T(a)\,$, in typical situations there is a
lower bound for the values $\,T\,,$ which depends (decreasingly)
only on the $\,X^l-$norm of $\,a\,,$ and not on $\,a\,$ itself.
Often, a weaker norm is sufficient to determine $\,T\,$.\par%
Note that Theorem \ref{abstract} is not the more general result that
one can get. We have preferred to avoid the full generality, since
our interest is mainly concerned with the inviscid limit result in
strong topologies. This means that if the initial data is given in a
Banach space $X$, the convergence result should be established in
$C([0,\,T];\,X)\,$. Actually, one can deduce that the
initial-boundary value problem \eqref{1;4}, \eqref{bcns} has a
unique solution in some class (such as $L^p(\,0,\,T;
W^{l,\,q}(\R^3_+))\,$) if the Cauchy problem has a unique solution
in the corresponding class, provided that the initial data
(prescribed in $\R^3_+$) satisfies the boundary conditions and the
compatibility conditions given in assumption \ref{sumo}.
\section{The inviscid limit.}
Vanishing viscosity limit results in $3-D$ domains, without boundary
conditions, have been studied by many authors. See, for instance,
\cite{const-foias}, \cite{kat}, \cite{K3}, \cite{Kato-ponce-cpam},
\cite{lionsM}, \cite{SW}, and the more recent papers \cite{bvkaz},
\cite{Masm}. In \cite{bvkaz}, \cite{K3} and \cite{Masm} results are
proved in the \emph{strong topology}. For results concerning
inviscid limits in non-smooth situations we refer to \cite{cw-2}.\par%
Concerning the vanishing viscosity problem in bounded domains, under
slip boundary conditions, we refer to \cite{bvcr},\cite{bvcr2},
\cite{crispo}, \cite{iftimie},  \cite{xin} and references
therein.\par%
In the particular 2-D case the assumption $\,\om \times \,n=\,0\,$
on $\,\Ga\,$ is simply replaced by $\,\om =\,0\,.$  For specific 2-D
vanishing viscosity results under slip-type boundary conditions we
refer to the classical papers, \cite{bardos},  \cite{judo},
\cite{mcgrath}, \cite{secchi}. See also the more recent papers
\cite{clopeau},  \cite{lopes}, and \cite{watanabe}.

\vspace{0.2cm}

In the following we consider the Navier-Stokes equations
\begin{equation}
\left\{
\begin{array}{l}
\partial_t\,\ccu^\nu +\,(\ccu^\nu \cdot\,\nabla)\,\ccu^\nu -\,\nu\,\Delta\,\ccu^\nu
+\,\nabla\,\cp^\nu=\,0,\\
\nabla \cdot\,\ccu^\nu =\,0\,,\\
\ccu^\nu(0)=\,\ca_\nu(x)\,,
\end{array}
\right. \label{1;4bistil}
\end{equation}
in $\,\R^3_+\,$, under the boundary condition
\begin{equation}
\left\{
\begin{array}{l}
\ccu^\nu\cdot \n=\,0,\\
\\ \widetilde{\om}^\nu \times \,\n=\,0\,.
\end{array}
\right. \label{bcnsbis}
\end{equation}
We assume that the positive viscosities $\nu$ are bounded from above
by an arbitrary, but fixed, constant.\par%
As already remarked, on flat portions of the boundary, \eqref{bcns}
coincides with the well known boundary condition
\eqref{boundary-ho}. In the 3-D problem, if the boundary is not
flat, it is not clear how to prove \emph{strong} inviscid limit
results. In fact, a substantial obstacle appears. See \cite{bvcr}
for some details on this point.\par%
Denote by $\,X^l(\R^3_+)\,$ the initial data's space. We want to
prove the convergence in $\,C([0,\,T\,]; \,X^l(\R^3_+)\,)\,$ of the
solutions $\,\ccu^\nu\,$, as the viscosity $\nu$ goes to zero (and,
possibly, $\ca_\nu$ converging to some $\ca$) to the solution
$\ccu^0$ of the Euler equations in $\,\R^3_+\,$
\begin{equation}
\left\{
\begin{array}{l}
\partial_t\,\ccu^0 +\,(\ccu^0\cdot\,\nabla)\,\ccu^0+\,\nabla\,\cp^0=\,0,\\
\na \cdot\,\ccu^0=\,0\,,\\
\ccu^0(0)=\,\ca(x)
\end{array}
\right. \label{1;4e}
\end{equation}
under the zero-flux boundary condition
\begin{equation}
\ccu^0\cdot\,\n=\,0\,.%
\label{bceu}
\end{equation}

\vspace{0.2cm}

Previous results, particularly related to ours, were proved in
\cite{xin}, in spaces $\,W^{3,\,2}$; in \cite{bvcr}, in spaces
$\,W^{2,\,p}$ and $\,W^{3,\,p}$, for any arbitrarily large $p$; and
in reference \cite{bvcr2}, in arbitrary $W^{k,\,p}$ spaces. However,
uniform convergence in time with values in the initial data space
$\,X\,$ is not proved. On the contrary, for the Cauchy problem, some
sharp vanishing viscosity limit results  are known. Theorem
\ref{abstract} allows immediate extension of these results to the
initial-boundary value problem. The following assumption is, in
fact, a condition on $X^l(\R^3)$. It requires that the vanishing
viscosity limit result holds in $X^l(\R^3)$ for the Cauchy problem
\begin{equation}
\left\{
\begin{array}{l}
\partial_t\,u^\nu +\,(u^\nu \cdot\,\nabla)\,u^\nu -\,\nu\,\Delta\,u^\nu +\,\nabla\,p^\nu=\,0,\\
\na \cdot\,u^\nu =\,0\,,\\
u^\nu(0)=\,a_\nu(x)\,.
\end{array}
\right. \label{1;4bis}
\end{equation}
\begin{assumption}
a) For each $\,\nu >\,0\,$ and each $\,a^\nu \in X^l(\R^3)\,$ the
Cauchy problem \eqref{1;4bis} admits a unique solution $\,u^\nu
\,\in C(\,[0,\,T]; \,X^l(\R^3)\,)\,,$ where $T>\,0\,$ is independent
of $\nu\,$ and of the particular $\,a^\nu \in X^l(\R^3)$, provided
that their norms are bounded from above by a given constant.
Furthermore, if the parameter $\,\nu\,$ tends to zero and $\,a^\nu
\,$ tends to $\,a\,$ in $\,X^l(\R^3)\,)\,,$ then $\,u^\nu \,$
converges in $\,C(\,[0,\,T]; \,X^l(\R^3)\,$ to the
unique solution $\,u^0\,$ of the Euler equations in $\,\R^3\,$.\par%
\label{coxi}
\end{assumption}

\begin{theorem}
Under the assumption \ref{coxi} one has the following result:\par%
Let the vector fields $\,\ca^\nu \in X^l(\R^3_+)\,$ satisfy the
boundary conditions \eqref{bcnsbis}, and the compatibility
conditions described in the assumption \ref{sumo}. Then, for each
$\,\nu >\,0\,$, the initial-boundary value problem \eqref{1;4bis},
\eqref{bcnsbis} admits a unique solution $\,\ccu^\nu \in
\,C(\,[0,\,T]; \,X^l(\R^3_+)\,)\,$. Furthermore, if the parameter
$\,\nu \,$ tends to zero (vanishing viscosity limit) and the initial
data $\,\ca^\nu\,$ converge to some $\,a\,$ in $\, X^l(\R^3_+)\,$,
then $\,\ccu^\nu \,$ converges in $\,C(\,[0,\,T];
\,X^l(\R^3_+)\,)\,,$ to the unique solution $\,\ccu^0\,$ of the
Euler
equations \eqref{1;4e} under the boundary condition \eqref{bceu}.%
\label{teoappl}%
\end{theorem}
The result follows from Theorem \ref{abstract} together with the
assumption \ref{coxi}. We define $\,a^\nu=\,\cT \,\ca^\nu\,$ as
being the mirror-images of the $\,\ca^\nu$'s, and $\,u^\nu\,\in
C(\,[0,\,T]; \,X^l(\R^3)\,)\,$ as being the solutions to the Cauchy
problems \eqref{1;4bis} with viscosity $\,\nu\,$ and initial data
$\,a^\nu \,$. Then, the solutions $\,\ccu^{\nu}\,$ of the boundary
value problems are the restrictions to the half-space $\,\R^3_+\,$
of the solutions $\,u^\nu\,$, and  $\,\ccu^0\,$ is the restriction
to the half-space $\,\R^3_+\,$ of $\,u^0\,$.\par%

Additional regularity and convergence results for the solutions
$\,u^\nu\,$, their time-derivatives, and pressure follow immediately
from corresponding results proved for the Cauchy problem.%

\vspace{0.2cm}

To apply the above theorem to a specific problem, we simply replace
the assumption \ref{coxi} by the known, desired, vanishing viscosity
result for the Cauchy problem. For instance, let us show an
application of the above theorem, in the case
$\,X^l(\R^3)=\,\,H^{l,\,2}_\si(\R^3)\,$.
\begin{theorem}
Assume that the initial data $\,\ca_\nu \in
\,H^{l,\,2}_\si(\R^3_+)\,$ satisfy the boundary condition
\eqref{bcnsbis}. Further, if $\,l \geq\,4\,$, assume that for each
\emph{odd} integer $k \in \,[3,\, l-\,1\,]\,$, the compatibility
condition
\begin{equation}
\,\pa^k_3 \,\ca_j =\,0\,, \quad \textrm{on} \quad \Ga\,, \quad
\textrm{for} \quad \,j=\,1,\,2%
\label{sim}
\end{equation}
holds. Then the initial-boundary value problem \eqref{1;4bis},
\eqref{bcnsbis} admits a unique solution $\,\ccu^\nu \in \,
C(\,[0,\,T]; \,H^{l,\,2}_\si(\R^3_+))\,.$ Furthermore, if
$$
\ca_\nu \rightarrow \,\ca\,, \quad \textrm{in} \quad
H^{l,\,2}_\si(\R^3_+)\,,
$$
as $\,\nu \rightarrow \,0\,,$ then, as $\,\nu \rightarrow \,0\,,$
\begin{equation}
\ccu^\nu \rightarrow \,\ccu\,, \quad \textrm{in} \quad C(\,[0,\,T];
\,H^{l,\,2}_\si(\R^3_+))\,,%
\label{coh}
\end{equation}
where $\,\ccu\,$ is the solution to the Euler equations \eqref{1;4e}
under the boundary condition \eqref{bceu}.%
\label{teoka}
\end{theorem}
For instance, if $l=\,3$ one has $l_0=\,2\,$. Hence the vanishing
viscosity limit holds in the space $\,H^3_\si(\R^3_+)\,$
\emph{without} assuming compatibility conditions on the initial
data. If $\,X^5(\R^3_+)=\,H^5_\si(\R^3_+)\,$, one has $l_0=\,4\,$.
Hence we have to assume the compatibility
condition $\,\pa^3_3 \,\ca_j=\,0\,,$ for $j=\,1,\,2\,.$\par%
In Theorem \ref{teoka} the assumption \ref{coxi} holds, as follows
essentially from results due to T. Kato, see \cite{kat} and
\cite{K3}. A simpler proof is shown by N. Masmoudi, see Theorem 2.1
in reference \cite{Masm}. We may also appeal to \cite{bvkaz}, to
prove the above result in the cubic domain case.\par%

As another application, we may extend to the boundary value problem
the results proved by T.Kato and G. Ponce in \cite{Kato-ponce-cpam},
where convergence of Navier-Stokes to Euler follows in Lebesgue
spaces $\,L^p_s(\,\R^n)\,$. These spaces are similar to
$\,W^{s,\,p}\,$ spaces. See also \cite{Kato-ponce-duke}. It would be
redundant to state here other specific results. Checking this
possibility case by case is, on the whole, an easy task.
\begin{remark}
In this section we have appealed to the Theorem \ref{abstract} for
extending to the boundary value problem the properties described in
assumption \ref{coxi} for the Cauchy problem. If we want to extend
different properties, we merely have to replace the assumption
\ref{coxi} by an assumption describing the corresponding properties
for the Cauchy problem.
\end{remark}

\section{Regularity for shear-thickening stationary flows.}
In this section we consider the following stationary system
describing the motion of a non-Newtonian fluid:\be
\begin{cases}\vspace{1ex}
-\nabla \cdot S\,(\cD\, u) +\nabla \pi =\,f
\,,\\%
\nabla\,\cdot\, u=\,0\,.
\end{cases}
\label{NNF}\ee We assume that the ``extra stress'' $S$ is given by
\be\label{ess}S(\,\cD\, u)=\,(\,\nu_0+\nu_1|{\mathcal{D}}
u|^{p-2}\,)\,{\mathcal{D}} u\,,\ee
 where ${\mathcal{D}} u$ is the symmetric
gradient of $u$, i.e.
$${\mathcal{D}}\,u=\,\frac 12 \left(\,\nabla\,u+\,\nabla\,u^T\,\right),$$
$\nu_0$, $\nu_1$  are positive constants, and $p>2$. Here, and in
the next section, we sacrifice a greater generality to emphasizing
the main ideas. Thus  \eqref{ess} and \eqref{exs2} (see below) are
just the canonical representative of a wider class of extra stress
tensors to which our proof applies.\par%
As done for the initial boundary value problem \eqref{1;4}, we draw
new results for the system \eqref{NNF} under the slip-boundary
conditions \eqref{bcns} from the corresponding known results for the
whole space. We are mainly interested in regularity results up to
the boundary. This problem has received various contributions in
recent years. The main open problem is to prove the
$L^2$-integrability, up to the boundary, of the second derivatives
of the solutions, in both the cases $p<2$ and $p>2$. In reference
\cite{bvlali} the half-space case $\R^3_+$ is considered, under slip
(and non-slip) boundary conditions, and $\,p>\,2\,$. The author
shows that the second ``tangential'' derivatives belong to
$L^2(\R^3_+)$, while the second ``normal'' derivatives belong to
some $L^l_{loc}(\overline{\R^3_+})$, for a suitable $l<2$. See
\cite{BKR2} and \cite{BKR} for recent, and more general, related
results (under the non-slip boundary condition), and for references.\par%
In the sequel the reflection technique enables us to improve the
regularity results, by overcoming the loss of regularity from the
tangential to the normal direction. See Theorem \ref{teoapplNN}
below.
\par Following the notation in \cite{bvlali}, we define
$\wt D^1(\R^3_+):=D^{1,2}(\R^3_+)$ as the completion of
$C_0^{\infty}(\overline {\R^3_+})$ with respect to the norm
$\|\nabla \,v\|\,,$ and set
$$
\wt V_2(\R^3_+)=\left\{\,v\in \, \wt D^{1}(\R^3_+):\,\nabla\cdot
v\,=\,0\,,\, v_3|_{x_3=0}\,=\,0\,\right\},
$$
endowed with the norm $\,\|\,\nabla \,v\,\|\,.$ We denote by $(\wt
V_2(\R^3_+))'\,$ the dual space of $\wt V_2(\R^3_+)$. Finally we set
$$\wt V(\R^3_+)=\left\{\,v\in \wt V_2(\R^3_+):\|\,\cD\, v\,\|_p<\infty\,\right\}\,,$$
endowed with the norm $\|\,\nabla \,v\,\|\,+\,\|\,\cD\, v\,\|_p\,$.
We use $$D^{1}(\R^3)\,,\ V_2(\R^3)\,,\  (V_2(\R^3))'\,, \
V(\R^3)\,,$$ for the corresponding spaces in $\R^3$\,.\par Assume
that $f \in (V_2)'\,$.
 We say that $u\,$ is a {\rm{weak solution}} of system
\eqref{NNF} if $u\in V$ satisfies \be\label{buf2} \frac
12\int_{\R^3}\left(\,\nu_0+\nu_1\,|{\mathcal{D}}
u|^{p-2}\,\right){\mathcal{D}} u\cdot {\mathcal{D}}\, v dx
=\,\int_{\R^3}\ f \cdot v \,dx\,, \ee for all $v \in \,V$. A
corresponding definition holds for the boundary value problem.\par%
We start by recalling the following result (for a sketch of the
proof, see below).
\begin{proposition}
For each  $\,f \in (V_2(\R^3))'\cap L^2(\R^3)\,$, the system
\eqref{NNF} in $\,\R^3\,$ admits a unique weak solution $\,u \,\in
V(\R^3)\,$. Furthermore, the derivatives
$\,D^2\,u\,$ belong to $L^2(\R^3)\,$.\par%
\label{coxistat}
\end{proposition}
By the above proposition, and thanks to the procedure developed in
the previous sections, one has the following theorem.
\begin{theorem}
Let be $\,\wt f \in (\wt V_2(\R^3_+))'\cap L^2(\R^3_+)\,$. Then, the
boundary value problem \eqref{NNF}, \eqref{bcns} in $\R^3_+$ admits
a unique weak  solution $\,\ccu \in \,\wt V(\R^3_+)\,$. Furthermore,
the derivatives $\,D^2\,\ccu\,$ belong to $L^2(\R^3_+)\,$.%
\label{teoapplNN}%
\end{theorem}
\begin{proof}
Actually, we merely have to check that  if $u$ is a solution of
\eqref{NNF} in $\R^3\,$ so is  $\gT\,u$, where $\gT$ is given by
Definition \ref{ttrans}. This last property is immediate, since
$$ (\,\cD\, (\gT\,u))_{ij}(x)=\left\{\begin{array}{ll} \,(\,\cD
\,u)_{ij}(\overline x),\ \mbox{ as } i=\,j=\,3\  \mbox{ or } \
i,\,j\,\in \{1,\,2\}\,,\\ \\
 -\,(\,\cD \,u)_{ij}(\overline x),\ \mbox{
as } i=\,3\,,  \,j\,\in \{1,\,2\}\,.
\end{array} \right .
 $$
Note that change of sign occurs if the index $\,3\,$ appears an odd
number of times. It is worth noting that, due to the low regularity
of the force term, we do not have to require any extra assumption,
like assumption \ref{numas}.\par%
As for Proposition \ref{coxistat}, the existence and uniqueness of
weak solutions in $\R^3$ are well known, and derive from basic
theory of monotone operators. See \cite{lions}. As far as the $L^2$
regularity is concerned, we note that the results in \cite{MNR}
immediately show that the derivatives $D^2\, u\,$ belong to
$L^2_{loc}(\R^3)\,$. By appealing to our reflection technique
results, this yields $\,D^2\, \ccu\,\in
\,L^2_{loc}(\overline{\R^3_+})\,.$ Here the restriction ``local''
merely means ``at finite distance''. This restriction is not
substantial, since it is formally due to the fact that in \cite{MNR}
the authors consider a bounded domain. Clearly, one gets ``global''
regularity in $\,\R^3\,$ by appealing to the Nirenberg's translation
technique in all the space (as in \cite{bvlali}, Lemma 4.1). Note
that in \cite{bvlali} translations are allowed only in the
tangential directions, since the problem is considered in the
half-space. However, in $\,\R^3\,$, translations can be done in any
direction. Since we turn the problem in the half-space with
slip-boundary conditions into the problem in the whole space, the
solution $\,\ccu\,$ of the boundary value problem has second
derivatives in $L^2(\R^3_+)\,$, since it is the restriction to the
half-space of the solution $\,u\,$, with $\,D^2\,u\,\in
L^2(\R^3)\,$.
\end{proof}
\section{The evolution shear-thinning problem in the ``periodic cube''.}
As announced in the introduction, the technique followed for the
half-space applies for other domains with flat boundaries. Here we
show in which way one can extend the procedure to problems in the so
called ``periodic cube'', with slip boundary conditions on two
opposite faces which make up here the significant boundary
$\,\Ga\,$), and periodicity on the remaining two pairs of faces.
This is by now a canonical situation, that allows to avoid
localization techniques and unbounded domains, hence to focus
attention on the main (the boundary value) problem. Consider the
half-cube
\be\label{omegatilde}\wt\Om=(-1,1)^2\times\left(-\frac12,\frac12\right)\,.\ee
We assume the slip boundary conditions (recall \eqref{asstil}) in
$\,\Ga=\,\{\,x:\, |\,x_1\,|\leq\,1,\,|\,x_2\,| \leq\,1,
\,x_3=\pm\frac12\, \}\,,$ namely
\begin{equation}
a_3(x)=\pa_3\,\wt a_1(x)=\pa_3\,\wt a_2(x)=0\,,
\mbox{if}\
x_3=\pm\frac12\,,%
\label{ot}
\end{equation}
together with periodicity in the $\,x_1\,$ and $\,x_2\,$ directions,
where $\wt a\in X^l(\wt\Om)$. Notation apes that used in the
previous sections for the half-space case, with the obvious
adaptations.\par%
Now we perform a mirror-extension (recall \eqref{difas} and
\eqref{defas}), on the upper and lower faces, from each $\,\wt a\in
X^l(\wt\Om)$ to a corresponding $a$, defined in the cube
$\Om=(-1,1)^3$, in the following way
\be\label{dbmirr}\left\{\ba{ll}\vs1a(x)=\,\wt a(x)\,,
&\mbox{if}\ x\in\,{\wt\Om}\,,\\
\vs1 a_i(x)=\,\wt a_i(x_1,\,x_2,\,1-x_3)\,,&\mbox{if}
\ x_3\in\,\left[\frac12,1\,\right]\ \mbox{and}\ i=1,2\,,\\
\vs1 a_3(x)=\,-\,\wt a_3(x_1,\,x_2,\,1-x_3)\,,&\mbox{if}
\ x_3\in\,\left[\frac12,1\,\right]\,,\\
\vs1 a_i(x)=\,\wt a_i(x_1,\,x_2,\,-1-x_3)\,,&\mbox{if}
\ x_3\in\,\left[-1,-\frac12\,\right]\ \mbox{and}\ i=1,2\,,\\
a_3(x)=\,-\,\wt a_3(x_1,\,x_2,\,-1-x_3)\,,&\mbox{if}
\ x_3\in\,\left[-1,-\frac12\,\right]\,.\\
\ea\right.\ee%
By imposing the compatibility conditions (similar to \eqref{numas})
$$\pa_3^k\,\wt a_1\left(x_1,\,x_2,\,\pm\textstyle\frac12\right)=
\pa_3^k\,\wt a_2\left(x_1,\,x_2,\,\pm\textstyle\frac12\right)=0\,,\
\ \forall\, k\in\NN,\ k\ \mbox{odd},\ k \in [\,3,\,l_0\,]\,,$$ we
get $a\in X^l(\Om)\,$. A direct computation shows that, for any
admissible multi-index $\al$,
$$D^\al a(x_1,x_2,-1)=D^\al a(x_1,x_2,1).$$
By using this procedure, we have transformed the original problem
with mixed boundary conditions in a purely periodic one. Likewise in
the half-space case, if we start from a suitable solution to the
totally periodic problem, its restriction to the half-cube turns out
to be a solution of the slip-periodic problem. Clearly, regularity
properties are preserved. As an application we show here an
existence and regularity result in the framework of the
shear--thinning fluids. The problem considered is the evolution of a
non--Newtonian fluid in a periodic cube with slip boundary
conditions, namely \be\label{periodicevolution}\left\{\ba{ll}
\pa_t\, u-\nabla\cdot S\,(\,\cD\, u\,)\,+(\,u\,\cdot\nabla\,)\,u+\nabla\, p=0\,,\\
\nabla\,\cdot\, u=\,0\,,\\
u(0,x)=\,\wt a(x)\,, \ea\right.\ee%
where, for the sake of simplicity, we assume that the extra stress
$S$ is of the following type \be S(\cD u)=(\,1+\,|\,\cD
u\,|)^{p-2}\,\cD\, u\,.\label{exs2}\ee%
\par
We want to solve problem \eqref{periodicevolution} in the cube $\wt
\Om$  with boundary conditions \eqref{ot} and $p<\,2$. We require
vanishing mean value, as usual in the space periodic case. Let us
introduce the spaces
$$V_p=\left\{\,v\in W^{1,p}(\Om), \,
\nabla\cdot v=0,\ v\,\mbox{ is } x-\mbox{periodic}\right\}\,,$$
$$\wt{V}_p=\left\{v\in W^{1,p}(\wt\Om), \,
 \nabla\cdot v=0,\ v\,\mbox{ is } (x_1,x_2)-\mbox{periodic},
 \, v_3=0 \ \mbox{if}\ x_3=\pm\frac12\right\}\,.$$
We start by recalling the following result.
\begin{proposition}
Let be $p\in\left(\frac 75,2\right)$. For each $\,a \in
\,W^{2,2}(\Om)\cap V_p\,$ the initial value problem
\eqref{periodicevolution} in $\Omega$ admits a unique solution
 $\,u \,\in L^\infty(0,T;\,V_p\,)\cap L^2(0,T;\,W^{2,2}(\Om))\,,$
 $ \pa_t u\in L^2(0,T;\,L^2(\Omega)\,)$, for some positive $T=\,T(a)$.\par%
\label{asspercube}
\end{proposition}
The existence of a solution is ensured
 by Theorem 17 in \cite{diru}, while its uniqueness
  follows by Corollary 18 and Theorem 19 in the same reference.
  See also \cite{BersDR} for the degenerate case $S(\cD u)=\,|\,\cD
u\,|^{p-2}\,\cD\, u\,$. \vspace{0.2cm}
\par One has the following theorem.
\begin{theorem}
Let be $p\in\left(\frac 75,2\right)$. Let the vector field $\,\ca
\in W^{2,2}(\wt\Om)\cap \wt V_p\,$ satisfy the boundary conditions
\eqref{ot}. Then, the initial-boundary value problem
\eqref{periodicevolution}, \eqref{ot} in $\wt \Omega$ admits a
unique solution $\,\ccu \in L^\infty(0,T;\,\wt V_p\,)\cap
L^2(0,T;\,W^{2,2}(\wt\Om))$ and $\pa_t \ccu\in
L^2(0,T;\,L^2(\wt\Om))$.%
\label{teopev}
\end{theorem}
The proof follows by extending the initial datum $\wt a$ by means of
equations  \eqref{dbmirr}, and by observing that if $u$ solves
system \eqref{periodicevolution} then $\gT\, u$ solves the same
system (see the previous section).
\par
We remark that in the present case the compatibility conditions (see
assumption \ref{sumo}) are not needed, as stated in Proposition
\ref{menosba}. \vskip 0.5cm \indent
 {\bf Acknowledgments}\,:
 The work of the second author was
supported by INdAM (Istituto Nazionale di Alta Matematica) through
a Post-Doc Research Fellowship.

\end{document}